\documentclass[11pt]{amsart}
\usepackage{amsmath,amsfonts,amsthm,amssymb,amscd, verbatim, graphicx}
\def\classification#1{\def\@class{#1}}
\classification{\null}
\DeclareFontFamily{OT1}{rsfs}{}
\DeclareFontShape{OT1}{rsfs}{n}{it}{<-> rsfs10}{}
\DeclareMathAlphabet{\mathscr}{OT1}{rsfs}{n}{it}

\newcommand{\Fq}{\mathbb{F}_q}
\newcommand{\mindeg}{\mathrm{mindeg}}
\newcommand{\minclass}{\mathrm{minclass}}

\newtheorem{prop}{Proposition}[section]
\newtheorem{thm}[prop]{Theorem}
\newtheorem{conj}[prop]{Conjecture}
\newtheorem{cor}[prop]{Corollary}
\newtheorem{lem}[prop]{Lemma}

\theoremstyle{definition}
\newtheorem{remark}[prop]{Remark}
\numberwithin{equation}{section}
\begin{document}

\title{On the product decomposition conjecture for finite simple groups}

\author{N. Gill, L. Pyber, I. Short, E. Szab\'o}
\address{L\'aszl\'o Pyber and Endre Szab\'o\newline
A. R\'enyi Institute of Mathematics\newline
Hungarian Academy of Sciences\newline
P.O. Box 127\newline
H-1364 Budapest}
\email{pyber@renyi.hu}
\email{endre@renyi.hu}

\address{Nick Gill and Ian Short\newline
Department of Mathematics and Statistics\newline
The Open University\newline
Milton Keynes, MK7 6AA\newline
United Kingdom}
\email{n.gill@open.ac.uk}
\email{i.short@open.ac.uk}
\thanks{L.P. is supported in part by OTKA NK78439 and K84233}
\thanks{E.Sz. is supported in part by OTKA NK81203 and K84233}
\thanks{N.G. would like to thank Harald Helfgott for allowing the use of his research funds to facilitate a visit to the R\'enyi Institute during which work on this paper was initiated. He would also like to thank the University of Bristol to which he has been a regular visitor over the course of writing this paper.}

\subjclass[2010]{20D06, 20D40, 20G40}
\keywords{Conjugacy, doubling lemma, product theorem, simple group}

\begin{abstract}
  We prove that if $G$ is a finite simple group of Lie type and $S$ a subset
  of $G$ of size at least two
  then $G$ is a product of at most $c\log|G|/\log|S|$ conjugates of $S$,
  where $c$ depends only on the Lie rank of $G$.
  This confirms a conjecture of Liebeck, Nikolov and Shalev
  in the case of families of simple groups of bounded rank. We also obtain various
related results about products of conjugates of a set within a group.
\end{abstract}

\maketitle

\section{Introduction}

Our starting point is the following conjecture of Liebeck, Nikolov and Shalev \cite{lns2}.

\begin{conj}\label{c: 2}
There exists an absolute constant $c$ such that if $G$ is a finite simple group and $S$ is a subset of $G$ of size at least two, then $G$ is a product of $N$ conjugates of $S$ for some $N\leq c \log|G|/ \log |S|$.
\end{conj}

Note that we must have $N\ge\log|G|\big/\log|S|$ by order considerations, and so the bound above is best possible up to the value of the constant $c$.

The conjecture is an extension of a deep (and widely applied) theorem of
Liebeck and Shalev. Indeed, the main result of \cite{lieshal} states that the
above conjecture holds when $S$ is a conjugacy class or, more generally, a
normal subset (that is, a union of conjugacy classes) of $G$.
In \cite{lns2} Conjecture~\ref{c: 2} is also proved for
sets of bounded size.

Somewhat earlier Liebeck, Nikolov and Shalev \cite{lns} posed the following (still unproved) weaker conjecture.

\begin{conj}\label{c: 1}
There exists an absolute constant $c$ such that if $G$ is a finite simple group and $H$ is any nontrivial subgroup of $G$, then $G$ is a product of $N$
conjugates of $H$ for some $N\leq c \log|G|/ \log |H|$.
\end{conj}

Conjecture \ref{c: 1} itself represents a dramatic generalization of a host of earlier work on product decompositions of finite simple groups, most of which prove Conjecture \ref{c: 1} for particular subgroups $H$.  For instance, in \cite{liepy} it is proved that a finite simple group of Lie type in characteristic $p$ is a product of 25 Sylow $p$-subgroups (see also \cite{bnp} for a recent improvement from 25 to 5).

Further positive evidence for Conjecture \ref{c: 1} is provided by \cite{lns3}, \cite{lub} and \cite{nikolov} (when $H$ is of type $SL_n$). Certain results of this type are essential to prove that finite simple groups can be made into expanders (see the announcement \cite{kln}).

The main purpose of this note is to prove Conjecture \ref{c: 2} for finite simple groups of Lie type of bounded rank. Put another way, we prove a version of Conjecture \ref{c: 2} in which the constant $c$ depends on the rank of the group $G$. Our main result follows.

\begin{thm}\label{t: gps}
Fix a positive integer $r$. There exists a constant $c=c(r)$ such that if $G$ is a finite simple group of Lie type of rank $r$ and $S$ is a subset of $G$ of size at least two then $G$ is a product of $N$ conjugates of $S$ for some $N\leq c \log|G|/ \log |S|$.
\end{thm}

In \cite{lns2} a weaker bound of the form $N\leq\big(\log|G|/ \log |S|\big)^{c(r)}$ is obtained. Also, in \cite{lns}, Theorem~\ref{t: gps} is proved when $S$ is a maximal subgroup of $G$.

As a byproduct of our proof we obtain two results of independent interest. In these results, and throughout the paper, we denote by $S^g$ the conjugate $g^{-1}Sg$ of a subset $S$ of a group $G$ by an element $g$ of $G$, and, given a positive integer $m$, we denote by $S^m$ the product $SS\dotsb S$ of $m$ copies of $S$. There should be no confusion between these two similar notations because the type of the exponent will always be given.

\begin{thm}\label{t: 2}
Fix a positive integer $r$. There exists a positive constant $\varepsilon=\varepsilon(r)$ such that if $G$ is a finite simple group of Lie type of rank $r$ and $S$ is a subset of $G$ then for some $g$ in $G$ we have $\left|S S^g\right|\ge|S|^{1+\varepsilon}$ or $S^3=G$.
\end{thm}

The next theorem is similar, but concerns only normal subsets, in which case we obtain absolute constants.

\begin{thm}\label{t: normal subsets}
There exists $\varepsilon>0$ and a positive integer $b$ such that if $G$ is a finite simple group and $S$ is a normal subset of $G$ then $\left|S^2\right|\ge|S|^{1+\varepsilon}$ or $S^b=G$.
\end{thm}

Theorem~\ref{t: normal subsets} relates to a result of Shalev \cite[Theorem 7.4]{shalev}, which we strengthen in Section~\ref{s: shalev}.

Note that the theorem would not be true were we to consider sets that are not normal. For instance, take $S$ to be a maximal parabolic subgroup in $G=PSL_n(q)$ with index $\frac{q^n-1}{q-1}$. Clearly $S^b=S$ for all positive integers $b$; on the other hand, for any positive number $\varepsilon$, and any $g$ in $G$, we have $|S S^g|\leq |G|\leq |S|^{1+\varepsilon}$ once $n$ is large enough. We conclude that neither of the given options can hold in this more general situation.

Theorems~\ref{t: 2} and \ref{t: normal subsets}, and the remarks of the previous paragraph, lead us to make the following conjecture.

\begin{conj}\label{c: 3}
There exists $\varepsilon>0$ and a positive integer $b$ such that if $S$ is a subset of a finite simple group $G$ then for some $g$ in $G$ we have $\left|S S^g\right|\ge|S|^{1+\varepsilon}$ or $G$ is the product of $b$ conjugates of $S$.
\end{conj}

Note that, by Theorems~\ref{t: gps}~and~\ref{t: 2}, Conjectures~\ref{c: 2},~\ref{c: 1}~and~\ref{c: 3} hold for all exceptional simple groups. Note too that all three conjectures could be phrased in terms of \emph{translates} of the set $S$, rather than conjugates. This follows from the simple fact that a product of translates of $S$ is equal to a translate of a product of conjugates of $S$. Similarly a product of conjugates of a translate of $S$ is equal to a translate of a product of conjugates of $S$, a fact which will be useful in its own right.

It is possible that Conjecture~\ref{c: 3} actually holds with $b=3$. When
$b=2$ counterexamples are given by large non-real conjugacy classes (see the
 final section of \cite{shalev} for some related issues).
Further counterexamples are given by
certain
families of maximal subgroups (see for example \cite[Corollary 2]{lps},
which states that large enough simple unitary groups of odd dimension cannot
be decomposed into the product of two proper subgroups).

We derive Theorems~\ref{t: gps} and \ref{t: 2} as consequences of the recent
Product theorem for finite simple groups, proved independently by Breuillard,
Green and Tao \cite{bgt2}, and Pyber and Szab\'o \cite{ps2} (see Section~\ref{s: proof2}). Theorem~\ref{t: normal subsets} follows from a version of Conjecture~\ref{c: 2} for normal subsets due to Liebeck and Shalev \cite{lieshal} and an extension of Pl\"unnecke's theorem \cite[Theorem 6.27]{taovu} to normal subsets of nonabelian groups (see Section~\ref{s: doubling}).

In the final section we use a result of Petridis \cite{petridis} to derive an analogue of the classical Doubling lemma, a special case of Pl\"unnecke's theorem. We refer to the new result as the Skew doubling lemma; it can be thought of as a nonabelian version of the classical Doubling lemma. The Skew doubling lemma is applied to prove that Conjecture~\ref{c: 2} implies Conjecture~\ref{c: 3}. In the other direction, a standard argument (similar to the proof of Corollary~\ref{c: polynomial}) shows that Conjecture~\ref{c: 3} implies that a simple group $G$ is a product of $\left(\log|G|/\log|S|\right)^c$ conjugates of $S$, a weaker version of Conjecture~\ref{c: 2}.

\section{Proof of Theorem~\ref{t: 2}}\label{s: proof2}

We begin with a result of Petridis \cite[Theorem 4.4]{petridis}, which extends work of Helfgott, Ruzsa and Tao \cite{helfgott3,Ru2009,Ru2010,tao}. It relates to the Doubling lemma for abelian groups, which we return to in Section~\ref{s: doubling}.

\begin{lem}\label{l: petridis}
Let $S$ be a finite subset of a group $G$. Suppose that there exist positive numbers $J$ and $K$ such that $|S^2|\leq J|S|$ and $|SgS|\leq K|S|$ for each $g$ in $S$. Then $|S^3| \leq J^7K|S|$.
\end{lem}

Suppose now that $G$ is a finite group, and let $\minclass(G)$ denote the size of the smallest nontrivial conjugacy class in $G$. Let $\minclass(S,G)$ denote the size of the smallest nontrivial conjugacy class in $G$ that intersects $S$, and let $\mindeg(G)$ denote the dimension of the smallest nontrivial complex irreducible representation of $G$.

As observed in \cite{npy}, a result of Gowers \cite{gowers} implies the following.

\begin{prop}\label{p: gowers}
Let $G$ be a finite group and let $k=\mindeg(G)$. Take $S\subseteq G$ such that
$|S| \geq \frac{|G|}{\sqrt[3]k}.$ Then $G=S^3$.
\end{prop}

Now let $G=G_r(q)$ be a simple group of Lie type of rank $r$ over $\Fq$, the finite field of order $q$. We need some facts about $G$. The first result can be deduced, for example, from \cite[Tables 5.1 and Theorem 5.2.2]{kl}.

\begin{prop}\label{p: conjugacy size}
We have $q^r\leq \minclass(G)<|G|\leq q^{8r^2}$.
\end{prop}

\begin{prop}\label{p: landseitz}
Let $k=\mindeg(G)$. Then $|G|<k^{8r^2}$.
\end{prop}
\begin{proof}
We use the lower bounds on projective representations given by Landazuri and
Seitz \cite{landseitz}, allowing for the slight errors corrected in
\cite[Table 5.3.A]{kl}. For $G\neq PSL_2(q)$, we see that $k\geq q$, and so
the result follows from Proposition~\ref{p: conjugacy size}.

Now suppose that $G=PSL_2(q)$; then $|G|<q^3$ and $r=1$. For $q\geq 5$ and $q\neq 9$, $k=\frac1{(2,q-1)}(q-1)$ and it is clear that $k^8>q^3$. When $q=4$ we have $k=2$ and the result follows; likewise when $q=9$ we have $k=3$ and the result follows.
\end{proof}

The next result was obtained independently in
\cite{guralnick-kantor} and \cite{stein}.

\begin{prop}\label{p: stein}
Each finite simple group $G$ is $\frac32$-generated; that is, for any nontrivial element $g$ of $G$ there exists $h$ in $G$ such that $\langle g, h\rangle=G$.
\end{prop}

\begin{cor}\label{c: 32}
Let $G$ be a finite simple group and let $S$ be a subset of $G$ of size at least two. Then some translate of $S$ generates $G$.
\end{cor}
\begin{proof}
Let $u$ and $v$ be distinct elements of $S$. Since $G$ is $\frac32$-generated, there
exists $x$ in $G$ such that $\langle vu^{-1}, x\rangle=G$. Therefore the translate
$Su^{-1}x$, which contains $x$ and $vu^{-1}x$, generates $G$.
\end{proof}

The next result, the Product theorem, is our primary tool for proving
Theorems~\ref{t: gps} and \ref{t: 2}. Versions of this result can be found in \cite{bgt2,ps2}.
It was first proved by Helfgott for the groups $PSL_2(p)$ and $PSL_3(p)$ in \cite{helfgott2, helfgott3}.

\begin{thm}\label{t: generating sets}
Fix a positive integer $r$. There exists a positive constant $\eta=\eta(r)$ such
that, for $G$ a finite simple group of Lie type of rank $r$ and $S$ a
generating set of $G$,  either $S^3=G$ or $|S^3|\geq |S|^{1+\eta}.$
\end{thm}

We can now prove Theorem~\ref{t: 2}.

\begin{proof}[Proof of Theorem~\ref{t: 2}]
Given a positive integer $r$, let $\eta$ be the constant from Theorem~\ref{t: generating sets}. It suffices to prove Theorem~\ref{t: 2} for sets $S$ of size larger than some constant $L>1$ that depends only on $\eta$, because if $|S|<L$, and $S^3\neq G$, then, by the simplicity of $G$, there is an element $g$ of $G$ such that  $|SS^g|\geq |S|+1$, and $|S|+1\geq |S|^{1+\delta}$, where $\delta=\log(L+1)/\log L - 1$. In particular, we assume that $|S|\geq 8^{\frac{2}{\eta}}$, and we define $\varepsilon = \tfrac{1}{16}\min\left\{\eta,\tfrac{1}{24r^2}\right\}$.

Since $G$ is $\frac32$-generated, there exists an element $g$ of $G$ such that the set $T=S\cup \{g\}$ generates $G$. We can apply Theorem~\ref{t: generating sets} to  $T$ to conclude that either $T^3=G$ or $|T^3|\geq |S|^{1+\eta}$.

Now,  $T^3$ is the union of the eight sets $SSS$, $SSg$, $SgS$, $gSS$, $Sgg$, $gSg$, $ggS$ and $\{ggg\}$. Suppose that $|T^3|\geq |S|^{1+\eta}$. By the pigeon-hole principle at least one of the eight sets is larger than $\frac18 |S|^{1+\eta}$. We assumed earlier that $|S|\geq 8^{\frac{2}{\eta}}$, from which it follows that $\frac18 |S|^{1+\eta} > |S|^{1+\frac{\eta}{2}}$. Therefore one of the first seven of the eight sets is larger than $|S|^{1+\frac{\eta}{2}}$. All of these seven sets except $SSS$ are equal to a translate of the product of one or two conjugates of $S$, so if any of these have size at least $|S|^{1+\frac{\eta}{2}}$ then $|SS^h|\geq |S|^{1+\frac{\eta}{2}}$ for some element $h$ of $G$. If, on the other hand, $|SSS|\geq|S|^{1+\frac{\eta}{2}}$, then Lemma~\ref{l: petridis} (with $J=K=|S|^{\frac{\eta}{16}}$) implies that there is an element $h$ of $S\cup \{1\}$ with  $|SS^h|\geq |S|^{1+\frac{\eta}{16}}$. Thus in both cases there is an element $h$ with $|SS^h|\geq |S|^{1+\varepsilon}$.

The remaining possibility is that $T^3=G$. If $S^3\neq G$ then Proposition~\ref{p: gowers} implies that $|S|\leq |G|/\sqrt[3]{k}$ where $k=\mindeg(G)$. But Proposition~\ref{p: landseitz} gives that $|S|\leq |G|^{1-\frac{1}{24r^2}}$, and this implies, in particular, that $|T^3|=|G|\geq |S|^{1+\frac{1}{24r^2}}$. The argument of the previous paragraph applies again, to give $|SS^h|\geq |S|^{1+\varepsilon}$ for some element $h$.
\end{proof}

Note that we can immediately deduce the following result of \cite{lns}
(which we will use later).

\begin{cor}\label{c: polynomial}
Fix a positive integer $r$. There exists a constant $d$ such that if $G$ is a finite simple group of Lie type of rank $r$ and $S$ is a subset of $G$ of size at least two then $G$ is a product of $N$ conjugates of $S$ for some $N\leq 3 (\log|G|/ \log |S|)^d$.
\end{cor}
\begin{proof}
Let $\varepsilon$ be the constant from Theorem~\ref{t: 2}, and define $d=\log_{1+\varepsilon}2$. Let $M$ be the integer part of $\log_{1+\varepsilon} \frac{\log|G|}{\log|S|}$. Theorem~\ref{t: 2} implies that $G$ is the product of $3\cdot 2^M$ conjugates of $S$, and
\[
3\cdot 2^M \leq 3\left(\frac{\log|G|}{\log|S|}\right)^{d}.
\]
\end{proof}

The results in this section motivate a common generalisation of the Product
theorem, and Conjecture~\ref{c: 3}, for groups of Lie type.

\begin{conj}\label{c: 4}
There exists  $\varepsilon>0$ and a positive integer $b$ such that the following statement holds. For each integer $r$ there is a positive integer $c(r)$  such that if $G$ is a finite simple group of Lie type of rank $r$ and $S$ a generating set of $G$, then either $|SS^g|\ge|S|^{1+\varepsilon}$ for some $g\in S^{c(r)}$, or else $G$ is the product of $b$ conjugates $S^{g_1},\dots,S^{g_b}$, where $g_1,\dots,g_b\in S^{c(r)}$.
\end{conj}

It would be interesting to prove
Conjecture~\ref{c: 3} in the case when $S$ is a subgroup of $G$.
A rather general qualitative result in this direction was obtained by
Bergman and Lenstra \cite{bergman-lenstra}.
They show that if $H$ is a subgroup of a group $G$
satisfying $\big|H H^g\big|\leq K|H|$ for all $g$ in $G$,
then $H$ is ``close to'' some normal subgroup $N$ of $G$,
in the sense that $\big|H:H\cap N\big|$ and $\big|N:H\cap N\big|$
are both bounded in terms of $K$.

\section{Proof of Theorem~\ref{t: gps}}\label{s: proof1}

Given an element $g$ of a group $G$ we define
\[
g^G=\{g^h\,:\, h\in G\},
\]
and, for a subset $Z$ of $G$,
\[
Z^G = \{Z^h\,:\, h\in G\}.
\]
We begin the proof of Theorem~\ref{t: gps} with a simple combinatorial lemma, which enables us to deal with ``small'' sets.

\begin{lem}\label{l: get a big set}
Let $S$ be a subset of a finite group $G$.
There exist a positive integer $m$ and
$m$ conjugates of $S$ such that their product $X$ satisfies
$$|X|= |S|^m\geq \frac{\sqrt{\minclass(SS^{-1},G)}}{|S|}\geq
\frac{\sqrt{\minclass(G)}}{|S|}.$$
\end{lem}

\begin{proof}
Define $X_1=S$ and, if possible, choose an element $g$ of $G$ such that
$X_1^{-1}X_1\cap gSS^{-1}g^{-1}=\{1\}$. Define $X_2=X_1gSg^{-1}$. Notice that
if $x_L, x_R\in X_1, s_L, s_R\in S$, and $x_Lgs_Lg^{-1}=x_Rgs_Rg^{-1}$, then
$x_R^{-1}x_L=gs_Rs_L^{-1}g^{-1}$.
Hence $x_R^{-1}x_L\in X_1^{-1}X_1\cap gSS^{-1}g^{-1}$,
and so $x_L=x_R$ and $s_L=s_R$. It follows that
$|X_2|=|X_1||S|$. Now repeat this process with $X_2$ replacing $X_1$, and so
on.

The process terminates with a set $X$ of size $|S|^m$, which is a product of $m$ conjugates of $S$, and such that $|X^{-1}X\cap gSS^{-1}g^{-1}|\geq 2$ for all $g$ in $G$.

Let $T$ be a set of smallest possible size that intersects every conjugate of
$Z=SS^{-1}$ nontrivially, and write $t=|T|$.
Let $n=|G:N_G(Z)|$, the number of $G$-conjugates of $Z$. By the pigeonhole principle there exists an element $g$ of $Z$ that lies in at least $\frac{n}{t}$ different conjugates of $Z$. Let us count the set
\[
\Omega=
\big\{(g',Z') \in g^G\times Z^G \, \big| \, g'\in Z'\big\}
\]
in two different ways.

First, since every conjugate of $g$ lies in the same number of conjugates of
$Z$, we know that $|g^G| \frac{n}{t}\leq |\Omega|.$ On the other hand it
is clear that $|\Omega|\leq n|Z|$. Putting these together we obtain that
$|g^G|\frac{n}{t}\leq n|Z|$. Therefore
$$t\geq \frac{|g^G|}{|Z|}\ge\frac{\minclass(SS^{-1},G)}{|S|^2}$$
and using $|X|^2\geq |X^{-1}X|\geq t$
our statement follows.
\end{proof}

\begin{remark}
Lemma \ref{l: get a big set} and Proposition~\ref{p: conjugacy size} imply that if $G$ is a simple group of Lie type of rank $r$ and $S$ a subset of size less that $q^{r/4}$ then we have $\big|SS^g\big|=|S|^2$ for some $g$ in $G$.
\end{remark}

We are now ready to prove Theorem~\ref{t: gps}.

\begin{proof}[Proof of Theorem~\ref{t: gps}]
As observed above, a product of conjugates of a translate of $S$ is equal to
the translate of a product of conjugates of $S$. By Corollary~\ref{c: 32},
a translate of $S$ generates $G$. Therefore we
assume that $S$ generates $G$.

Suppose that $|S| \geq \big|\minclass(G)\big|^{1/4}$; then
 $|G|<|S|^{32r}$ by Proposition~\ref{p: conjugacy size}.
Now Corollary~\ref{c: polynomial} implies that $G$ is a product of fewer than
$3(32r)^d$ conjugates of $S$.
The theorem holds in this case with $c=3(32r)^d$.

Suppose instead that $|S| < |\minclass(G)|^{1/4}$. By Lemma~\ref{l: get a big
  set} we can choose conjugates $S_1,\dots,S_m$ of $S$ such that the set
$X=S_1\dotsb S_m$ satisfies $|X|=|S|^m$ and
\[
|X|\geq \frac{\sqrt{|\minclass(G)|}}{|S|} \ge
\big|\minclass(G)\big|^{1/4} \,.
\]
It follows from the first part of the proof that $G$ is a product of fewer than $c\log|G|/\log|X|$ conjugates of $X$. Therefore $G$ is a product of fewer than $mc\log|G|/\log|X|$ conjugates of $S$ and, since $\log|X|=m\log|S|$, the result follows.
\end{proof}

\section{Pl\"unnecke-Ruzsa estimates for nonabelian groups}\label{s: doubling}

The following basic result in additive combinatorics is due to Pl\"unnecke \cite{plun, plun2} (see also \cite[Section~6.5]{taovu}).

\begin{thm}\label{t: doubling}
Let $A$ and $B$ be finite sets in an abelian group $G$ and suppose that $|A B|\leq K|A|$ where $K$ is a positive number. Then for any positive integer $m$ there exists a nonempty subset $X$ of $A$ such that
\[
|X B^m|\leq K^m|X|.
\]
In particular, $|B^2|\leq K|B|$ implies that $|B^m|\leq K^m|B|$ for $m=1,2,\dotsc$.
\end{thm}

The last statement (``In particular\ldots'') is called the Doubling lemma; it does not hold for nonabelian groups, however, as we saw in Lemma~\ref{l: petridis}, there are useful analogues in this context due to Helfgott, Petridis, Ruzsa and Tao \cite{helfgott3,petridis,Ru2009,Ru2010,tao}. Petridis also proved the following lemma \cite[Proposition 2.1]{petridis}.

\begin{lem}\label{l: petridis 2}
Let $X$ and $B$ be finite sets in a group. Suppose that
\[
\frac{|XB|}{|X|} \leq \frac{|ZB|}{|Z|}
\]
for all $Z\subseteq X$. Then, for all finite sets $C$,
\[
|CXB|\leq \frac{|CX| |XB|}{|X|}.
\]
\end{lem}

Using this lemma we can extend Pl\"unnecke's theorem to normal subsets of nonabelian groups. The statement and proof mimic \cite[Theorem 3.1]{petridis}, which is a stronger version of Theorem~\ref{t: doubling}.

\begin{thm}\label{t: normal doubling}
 Let $A$ and $B$ be finite sets in a group $G$ with $B$ normal in $G$. Suppose that $|A B|\leq K|A|$ for some positive number $K$. Then there exists a nonempty subset $X$ of $A$ such that
\[
|X B^m|\leq K^m|X|
\]
for $m=1,2,\dotsc$. In particular, $|B^2|\leq K|B|$ implies that $|B^m|\leq K^m|B|$ for $m=1,2,\dotsc$
\end{thm}

\begin{proof}
 We proceed by induction on $m$. First choose $X\subseteq A$ such that
\[
\frac{|X B|}{|X|} \leq \frac{|Z B|}{|Z|}
\]
for all $Z\subseteq A$. Then
\[
|X B|\leq |X| \frac{|A B|}{|A|}\leq K|X|,
\]
so the result is true for $m=1$.

Now suppose that $|X B^{m}|\leq K^{m}|X|$ for some positive integer $m$. Normality of $B$ implies that $|X B^{m+1}|=|B^{m} X  B|$, and then Lemma~\ref{l: petridis 2} gives
\[
|X B^{m+1}| = |B^{m} X  B|\leq \frac{|B^mX||XB|}{|X|} \leq K^{m+1}|X|.
\]
This verifies the inductive step, and completes the proof of the theorem.
\end{proof}

Following an argument of Petridis (see the proof of \cite[Theorem 1.2]{petridis}) we
observe that the Pl\"unnecke-Ruzsa estimates \cite[Corollary 6.29]{taovu} can also be generalised using Theorem~\ref{t: normal doubling}.

\begin{cor}
Suppose that $A$ and $B$ are subsets of a group $G$, with $B$ normal in $G$, and $|AB|\leq K|A|$. Then
\[
|B^mB^{-n}|\leq K^{m+n}|A|
\]
for all positive integers $m$ and $n$.
\end{cor}

Theorem~\ref{t: normal doubling} suggests that certain techniques in additive combinatorics concerning subsets of abelian groups can be applied to normal subsets of nonabelian groups. The next example -- which is a consequence of Pl\"unnecke's theorem, and generalises \cite[Corollary 2.4]{Ru2009} -- supports this suggestion.

\begin{thm}\label{t: plun}
 Let $A$ and $B$ be subsets of a group $G$ with $B$ normal in $G$, and suppose that $|AB^j|\leq K|A|$ for some positive integer $j$. If $m\geq j$ then
\[
|B^m|  \leq K^{\frac{m}{j}}|A|.
\]
\end{thm}
\begin{proof}[Sketch of proof]
We use the notation of \cite[Section~6.5]{taovu}. Construct the $m$-tuple of directed bipartite graphs
$$(G_{A,B}, G_{A B, B}, \dots, G_{A B^{m-1}, B}).$$
This $m$-tuple is a Pl\"unnecke graph. Now Pl\"unnecke's theorem \cite[Theorem 6.27]{taovu} yields the result immediately.
\end{proof}

\section{Proof of Theorem~\ref{t: normal subsets}}\label{s: shalev}

In this section we prove Theorem~\ref{t: normal subsets} and generalise some related results of Shalev. We will need the following theorem of Liebeck and Shalev \cite{lieshal}.

\begin{thm}\label{t: normal}
There exists an absolute positive constant $a$ such that, if $G$ is a finite simple group and $S$ is a nontrivial normal subset of $G$, then $G=S^m$, where $m\leq a\frac{\log|G|}{\log|S|}$.
\end{thm}

\begin{proof}[Proof of Theorem~\ref{t: normal subsets}]
Let $a$ be the absolute constant from Theorem~\ref{t: normal}. Choose a positive integer $b$ larger than $2a$. Suppose first that $|S|\geq \sqrt{|G|}$. Then Theorem \ref{t: normal} implies that $G=S^m$ where
\[
m\leq \frac{a\log|G|}{\log|S|} \leq 2a \leq b,
\]
and hence $S^b=G$.

Now suppose that $|S|\leq \sqrt{|G|}$. Then
\[
\frac{\log|S|}{a\log|G|} \geq \frac{\log|S|}{2a(\log|G|-\log|S|)}=\frac{\log|S|}{2a(\log(|G|/|S|)}.
\]
Theorem \ref{t: normal} implies, once again, that for some $m\leq\frac{a\log|G|}{\log|S|}$ we have $G=S^m$. Hence, applying Theorem~\ref{t: normal doubling} to the normal subset $S$, we see that
\[
\frac{|S^2|}{|S|} \geq \left(\frac{|S^m|}{|S|}\right)^{\frac{1}{m}}
\geq \left(\frac{|G|}{|S|}\right)^{\frac{\log|S|}{a\log|G|}}
\geq \left(\frac{|G|}{|S|}\right)^{\frac{\log|S|}{2a(\log(|G|/|S|)}}
=|S|^{\frac{1}{2a}}\geq |S|^\frac{1}{b},
\]
and this completes the proof.
\end{proof}

The next result is a strengthening of \cite[Theorem 7.4]{shalev}.

\begin{prop}\label{p: shalev}
For every $\delta>0$ there exists $\varepsilon>0$ such that for any finite simple group $G$ and subsets $A$ and $B$ of $G$ with $B$ normal in $G$ and $|A|\leq |G|^{1-\delta}$ we have
\[
|AB|\geq |A||B|^{\varepsilon}.
\]
\end{prop}
\begin{proof}
We assume that $A$ is nonempty and $B$ is nontrivial, otherwise the result is immediate.

By Theorem~\ref{t: normal}, $G=B^m$, where $m\leq a\frac{\log|G|}{\log|B|}$. Let $K=|AB|/|A|$. Then, by Theorem~\ref{t: normal doubling}, there is a nonempty subset $X$ of $A$ such that $|XB^m|\leq K^m|X|$. It follows that
\[
|G|=|B^m|=|XB^m| \leq K^m|X|\leq K^m|A|.
\]
Since $|A|\leq |G|^{1-\delta}$ and $m\leq a\frac{\log|G|}{\log|B|}$ we can rearrange this inequality to give
\[
|G|^\delta \leq K^{a\frac{\log|G|}{\log|B|}}.
\]
This is equivalent to $|B|^\frac{\delta}{a}\leq K$, which, with $\varepsilon=\frac{\delta}{a}$, is the required result.
\end{proof}

Proposition~\ref{p: shalev} constitutes the expansion result for $B^2$ that was partially proven in \cite[Proposition 10.4]{shalev}. Furthermore it goes some way towards a proof of \cite[Conjecture 10.3]{shalev} although what remains is the more difficult part of the conjecture.

We can strengthen \cite[Proposition 10.4]{shalev} in a different direction as follows.

\begin{prop}\label{p: shalev 2}
For every $\delta>0$ and positive integer $r$ there exists $\varepsilon>0$ such that for any finite simple group $G$ of Lie type of rank $r$ and any set $S\subseteq G$ such that $|S|\leq |G|^{1-\delta}$, there exists $g$ in $G$ such that
$$|SS^g|\geq |S|^{1+\varepsilon}.$$
\end{prop}
\begin{proof}
Given $\delta>0$ and a positive integer $r$, let $\varepsilon$ be the positive constant from Theorem~\ref{t: 2}. Now choose any subset $S$ of $G$ such that $|S|\leq |G|^{1-\delta}$. According to Theorem~\ref{t: 2}, either $|S S^g|\geq |S|^{1+\varepsilon}$ or else $S^3=G$. In the former case the result is proven. In the latter case we apply Lemma~\ref{l: petridis} with $J=K=(|S^3|/|S|)^{1/10}$ to deduce the existence of an element $g$ of $G$ with $|SgS|> K|S|$. Then, using $S^3=G$ and $|G|\geq |S|^{1+\delta}$, it follows that
\[
|SgS| > \left(\frac{|S^3|}{|S|}\right)^{\frac{1}{10}}|S|\geq |S|^{1+\frac{\delta}{10}}.
\]
Provided that $\varepsilon$ is chosen to be smaller than $\tfrac{\delta}{10}$, the inequality $|S S^g|\geq |S|^{1+\varepsilon}$ is again satisfied.
\end{proof}

\section{The Skew doubling lemma}\label{s: skew doubling}

The next result is another analogue of the Doubling lemma for nonabelian groups, which we call the {\it Skew doubling lemma}.

\begin{lem}[Skew doubling lemma]\label{l: skew doubling}
If $S$ is a finite subset of a group $G$ such that, for some positive number $K$, $|S S^g|\leq K|S|$ for every conjugate $S^g$ of $S$, then
\[
|S_1\dotsb S_m|\leq K^{14(m-1)}|S|
\]
for $m=1,2,\dotsc$, where each of $S_1,\dots,S_m$ is any conjugate of either $S$ or $S^{-1}$.
\end{lem}

To prove Lemma~\ref{l: skew doubling} we will use Lemma~\ref{l: petridis} and the following result, Ruzsa's triangle inequality \cite{ruzsa} (see also \cite[Section~2.3]{taovu}).

\begin{lem}\label{l: ruzsa}
 Let $U$, $V$ and $W$ be finite subsets of a group $G$. Then
\[
\frac{|VW^{-1}|}{|U|} \leq \frac{|UV^{-1}|}{|U|} \frac{|UW^{-1}|}{|U|}.
\]
\end{lem}

First we prove a special case of Lemma~\ref{l: skew doubling}.

\begin{lem}\label{l: tao}
Let $S$ be a finite subset of a group $G$. Suppose that $K$ is a positive number such that $|SS^g| \leq K|S|$ for each $g$ in $G$. Then $|S_1S_2S_3| \leq K^{14}|S|$, where each of $S_1$, $S_2$ and $S_3$ is any conjugate of either $S$ or $S^{-1}$.
\end{lem}
\begin{proof}
Choose elements $a$ and $b$ of $G$.
We can apply Lemma \ref{l: petridis} with $J=K$ to obtain
\[
|S^3|\leq K^8|S|.
\]
Using this inequality and Lemma~\ref{l: ruzsa}
(with $U=S^{-1}$, $V=SS$ and $W=S$) we obtain
\[
\frac{|SSS^{-1}|}{|S|}\leq
\frac{|S^{-1}S^{-1}S^{-1}|}{|S|}\frac{|S^{-1}S^{-1}|}{|S|}=\frac{|S^3|}{|S|}\frac{|S^2|}{|S|}\leq
K^{9}.
\]
Using this inequality and Lemma~\ref{l: ruzsa} (with $U=S$, $V=S^{-1}$
and $W=SS^{-1}$) we obtain
\[
\frac{|S^{-1}SS^{-1}|}{|S|}\leq
\frac{|SS|}{|S|}\frac{|SSS^{-1}|}{|S|}\leq K^{10}.
\]
Using this inequality and Lemma~\ref{l: ruzsa} (with $U=S^{-1}$,
$V=SS^{-1}$ and $W=Sa$) we obtain
\[
\frac{|SS^{-1}a^{-1}S^{-1}|}{|S|}\leq
\frac{|S^{-1}SS^{-1}|}{|S|}\frac{|S^{-1}a^{-1}S^{-1}|}{|S|}\leq K^{11}.
\]
Using this inequality and Lemma~\ref{l: ruzsa} (with $U=S$, $V=SaS$
and $W=S^{-1}b^{-1}$) we obtain
\begin{equation}\label{e1}
\frac{|SaSbS|}{|S|}\leq
\frac{|SS^{-1}a^{-1}S^{-1}|}{|S|}\frac{|SbS|}{|S|}\leq K^{12}.
\end{equation}
Using this inequality and Lemma~\ref{l: ruzsa} (with $U=S$, $V=S^{-1}$
and $W=S^{-1}b^{-1}S^{-1}a^{-1}$) we obtain
\begin{equation}\label{e2}
\frac{|S^{-1}aSbS|}{|S|}\leq \frac{|SS|}{|S|}\frac{|SaSbS|}{|S|}\leq K^{13}.
\end{equation}
Finally, using this inequality and Lemma~\ref{l: ruzsa}
(with $U=S^{-1}$, $V=S^{-1}aSb$ and $W=S$) we obtain
\begin{equation}\label{e3}
\frac{|S^{-1}aSbS^{-1}|}{|S|}\leq
\frac{|S^{-1}b^{-1}S^{-1}a^{-1}S|}{|S^{-1}|}\frac{|S^{-1}S^{-1}|}{|S^{-1}|} =
\frac{|S^{-1}aSbS|}{|S|}\frac{|SS|}{|S|} \leq
K^{14}.
\end{equation}
Equations \eqref{e1}, \eqref{e2} and \eqref{e3} imply that, given any
conjugates $S_1$, $S_2$ and $S_3$ of either $S$ or $S^{-1}$, we have
$|S_1S_2S_3|/|S| \leq K^{14}$, as required.
\end{proof}

We need the following proposition.

\begin{prop}\label{p: top}
If $A$ and $B$ are finite subsets of a group $G$ such that, for some positive number $K$, $|BB^g|\leq K|B|$ for every conjugate $B^g$ of $B$, then
\[
|AB_1B_2|\leq K^{14}|AB_3|,
\]
where each of $B_1$, $B_2$ and $B_3$ is any conjugate of $B$ or $B^{-1}$.
\end{prop}
\begin{proof}
By Lemma~\ref{l: tao} we have
\[
\frac{|B_3^{-1}B_1B_2|}{|B_3|} \leq K^{14},
\]
where each of $B_1$, $B_2$ and $B_3$ is any conjugate of $B$ or $B^{-1}$. Applying Lemma~\ref{l: ruzsa} with $U=B_3^{-1}$, $V=A$ and $W=B_2^{-1}B_1^{-1}$ we obtain
\[
\frac{|AB_1B_2|}{|AB_3|}=\frac{|AB_1B_2|}{|B_3^{-1}A^{-1}|}\leq \frac{|B_3^{-1}B_1B_2|}{|B_3|}\leq K^{14},
\]
as required.
\end{proof}

We can finally prove Lemma~\ref{l: skew doubling}.

\begin{proof}[Proof of the Skew doubling lemma]
The result holds trivially when $m=1$ and $m=2$. Suppose that $m\geq 3$. Apply Proposition~\ref{p: top} with $B=S$, $A=S_1\dotsb S_{n-2}$, $B_1=B_3=S_{n-1}$  and $B_2=S_{n}$ to see that
\[
\frac{|S_1\dotsb S_{n}|}{|S_1\dotsb S_{n-1}|} \leq K^{14}
\]
for $n=3,4,\dots,m$. It follows that
\begin{align*}
\frac{|S_1\dotsb S_{m}|}{|S|} &= \left(\frac{|S_1\dotsb S_{m}|}{|S_1\dotsb S_{m-1}|}\right)\left(\frac{|S_1\dotsb S_{m-1}|}{|S_1\dotsb S_{m-2}|}\right)\dotsb \left(\frac{|S_1S_2S_3|}{|S_1S_2|}\right)\left(\frac{|S_1S_2|}{|S_1|}\right) \\
&\leq (K^{14})^{m-2}K \\
& \leq K^{14(m-1)},
\end{align*}
as required.
\end{proof}

Using the Skew doubling lemma we can derive Conjecture~\ref{c: 3} from Conjecture~\ref{c: 2}. The proof is similar to the proof of Theorem~\ref{t: normal subsets}.

\begin{proof}[Proof that Conjecture \ref{c: 2} implies Conjecture \ref{c: 3}]
Let $c$ be the absolute constant from Conjecture~\ref{c: 2}. We define
$b$ to be a positive integer greater than $2c$,  and
$\varepsilon=1/(28c)$. Suppose first that $|S|\geq \sqrt{|G|}$. Then
Conjecture~\ref{c: 2} implies that $G=S_1\dotsb S_N$, for conjugates
$S_1,\dots,S_N$ of $S$, where
\[
N\leq \frac{c\log|G|}{\log|S|} \leq 2c < b,
\]
and hence $G$ is certainly the product of $b$ conjugates of $S$.

Now suppose that $|S|\leq \sqrt{|G|}$. Then
\[
\frac{\log|G|-\log|S|}{c\log|G|-\log|S|} \geq
\frac{\log|G|-\log|S|}{c\log|G|} \geq \frac{1}{2c}.
\]
In particular observe that
\[c\log|G|-\log|S| \leq 2c(\log|G|-\log|S|) = 2c\log(|G|/|S|).\]
Conjecture~\ref{c: 2} implies, once again, that for some
$N\leq\frac{c\log|G|}{\log|S|}$ we have $G=S_1\dotsb S_N$, for
conjugates $S_1,\dots,S_N$ of $S$.
Using the Skew doubling lemma, Lemma~\ref{l: skew doubling}, we see
that there is an element $g$ of $G$ for which
\[
\frac{|SS^g|}{|S|} \geq \left(\frac{|S_1\dots
S_N|}{|S|}\right)^{\frac{1}{14(N-1)}}
\geq \left(\frac{|G|}{|S|}\right)^{\frac{\log|S|}{14(c\log|G|-\log|S|)}}
\geq \left(\frac{|G|}{|S|}\right)^{\frac{\log|S|}{28c(\log(|G|/|S|))}}
\geq |S|^{\frac{1}{28c}},
\]
and this completes the proof.
\end{proof}

\bibliographystyle{plain}
\bibliography{paper6}

\end{document}